\documentclass{amsart}

\usepackage[a4paper,hmargin=3.5cm,vmargin=4cm]{geometry}
\usepackage{amsfonts,amssymb,amscd,amstext}
\usepackage{graphicx}
\usepackage[dvips]{epsfig}
\usepackage{mathpazo}
\usepackage{enumerate}

\def\Ccal{\mathcal{C}}

\def\Ncal{\mathcal{N}}
\def\Mcal{\mathcal{M}}
\def\Hcal{\mathcal{H}}
\def\Xcal{\mathcal{X}}

\def\Ycal{\mathcal{Y}}
\def\dsf{\mathsf{d}}
\def\C{\mathbb{C}}\def\c{\mathbb{C}}
\def\d{\mathbb{D}}
\def\R{\mathbb{R}}\def\r{\mathbb{R}}
\def\z{\mathbb{Z}}
\def\n{\mathbb{N}}

\def\b{\mathbb{B}}\def\b{\mathbb{B}}

\newtheorem{theorem}{Theorem}[section]

\newtheorem{claim}[theorem]{Claim}
\newtheorem{lemma}[theorem]{Lemma}

\newtheorem{remark}[theorem]{Remark}
\newtheorem{definition}[theorem]{Definition}

\theoremstyle{definition}

\numberwithin{equation}{section}
\numberwithin{figure}{section}

\begin{document}

\title[Proper holomorphic embeddings of Riemann surfaces]
{Proper holomorphic embeddings of Riemann surfaces with arbitrary topology into $\c^2$}

\author[A.~Alarc\'{o}n]{Antonio Alarc\'{o}n}
\address{Departamento de Geometr\'{\i}a y Topolog\'{\i}a \\
Universidad de Granada \\ E-18071 Granada \\ Spain}
\email{alarcon@ugr.es}

\author[F.J.~L\'{o}pez]{Francisco J. L\'{o}pez}
\address{Departamento de Geometr\'{\i}a y Topolog\'{\i}a \\
Universidad de Granada \\ E-18071 Granada \\ Spain}
\email{fjlopez@ugr.es}


\thanks{Research partially
supported by MCYT-FEDER research project MTM2007-61775 and Junta
de Andaluc\'{i}a Grant P09-FQM-5088}


\begin{abstract}
We prove that given an open Riemann surface $\Ncal,$ there exists an open domain $\Mcal\subset \Ncal$ homeomorphic to $\Ncal$ which properly holomorphically embeds in $\c^2.$ Furthermore, $\Mcal$ can be chosen with hyperbolic conformal type. 
In particular, any open orientable surface $M$ admits a complex structure $\Ccal$ such that  $(M,\Ccal)$ can be properly  holomorphically  embedded into $\c^2.$
\end{abstract}

\maketitle

\thispagestyle{empty}

\section{Introduction}\label{sec:intro}

It is classically known that any open Riemann surface properly holomorphically embeds in $\c^3$ and immerses in $\c^2$ \cite{rem,Na1,nar,bis}. Bell-Narasimhan's conjecture asserts that any open Riemann surface can be properly holomorphically embedded in $\c^2$ \cite[Conjecture 3.7, p. 20]{BN}. Although this old embeddability problem has generated vast literature, it still remains open.

The first existence results for discs and annuli can be found in \cite{42} and \cite{32,5}, respectively. More recently, it has been proved that any finitely connected planar domain without isolated boundary points properly holomorphically embeds into $\c^2$ \cite{24} (see also \cite{CG}). Furthermore, any open orientable surface of finite topology admits a complex structure properly holomorphically embedding in $\c^2$ \cite{CF}. In the last few years, this area has experimented a great grothw. Specially interesting are the works by Wold \cite{W1,W2} and Forstneri${\rm \breve{c}}$ and Wold \cite{FW} (see also \cite{Ma}). These authors have shown that any bordered Riemann surface whose closure admits a (non-proper) holomorphic embedding into $\c^2$ actually properly holomorphically embeds  into $\c^2.$ (A bordered Riemann surface is the interior of a compact one-dimensional complex manifold with smooth boundary consisting of finitely many closed Jordan arcs.)
In  all these constructions, the (finite) topological type of the surface, and even its conformal structure, is not changed during the process. 

The aim of this paper is to show that the topology of an open Riemann surface plays no role in this setting. We extend the above mentioned result by $\breve{\text{C}}$erne and Forstneri$\breve{\text{c}}$ \cite{CF} to the case of surfaces with arbitrary topology, proving the following topological version of Bell-Narasimhan's conjecture:

\begin{quote}
{\bf Main Theorem.} {\em Let $\Ncal$ be an open Riemann surface. 

Then there exists an open domain $\Mcal\subset \Ncal$  homeomorphic to $\Ncal$ carrying a proper holomorphic embedding $\Ycal:\Mcal\to\c^2.$

In particular, any open orientable surface $M$ admits a complex structure $\Ccal$
such that the Riemann surface $(M,\Ccal)$ properly holomorphically embeds in $\c^2.$}
\end{quote}

The proper embedding $\Ycal:\Mcal\to\c^2$ in Main Theorem is obtained as the limit of a sequence of holomorphic embeddings $\{Y_n :M_n\to\c^2\}_{n\in\n},$ where $\{M_n\}_{n\in\n}$ is a suitable expansive sequence of compact regions in $\Ncal$ and $\Mcal=\cup_{n\in\n} M_n.$  The sequence is constructed by combining a bridge principle for holomorphic embeddings with Forstneri${\rm \breve{c}}$ and Wold's techniques. 


It is worth mentioning that the open Riemann surface $\Mcal$ in Main Theorem can be chosen of hyperbolic conformal type. Finally, let us point out that Main Theorem actually follows from a more general extension result for holomorphic embeddings into $\c^2$ (see Theorem \ref{th:main}). 


\section{Preliminaries}

As usual, we denote by $\|\cdot\|$ the Euclidean norm in $\c^n,$ $n\in\n,$ and for any compact topological space $X$ and continuous map $f:X \to \C^n$ we denote by $\|f\|=\max\{\|f(p)\|\,|\, p \in X\}$ the maximum norm of $f$ on $X.$

Non-compact Riemann surfaces without boundary are said to be {\em open}.

\begin{remark}\label{re:fun}
Throughout this paper, $\Ncal$ and $\omega$ will denote a fixed but arbitrary  open Riemann surface and a complete smooth conformal metric on $\Ncal,$ respectively.
\end{remark}

For any  $S \subset \Ncal,$ $S^\circ$ and $\overline{S}$ will denote the interior and the closure of $S$ in $\Ncal,$ respectively.
 
Given a Riemann surface $M$ contained in $\Ncal,$  we denote by $\partial M$ the $1$-dimensional topological manifold determined by its boundary points. Open connected subsets of $\Ncal$ will be called {\em domains}, and those proper topological subspaces of $\Ncal$ being Riemann surfaces with boundary are said to be  {\em regions}.

A subset $S \subset \Ncal$ is said to be {\em Runge} if $\Ncal-S$ has no relatively compact components in $\Ncal,$ or equivalently, if the inclusion map $\jmath_S: S\hookrightarrow \Ncal$ induces a group monomorphism  $(\jmath_S)_*:\Hcal_1(S,\z) \to \Hcal_1(\Ncal,\z).$ In this case we identify the groups  $\Hcal_1(S,\z)$ and  $(\jmath_S)_*(\Hcal_1(S,\z)) \subset \Hcal_1(\Ncal,\z)$ via $(\jmath_S)_*$ and consider $\Hcal_1(S,\z) \subset \Hcal_1(\Ncal,\z).$ 

Two Runge subsets $S_1,$ $S_2\subset \Ncal$ are said to be {\em isotopic} if $\Hcal_1(S_1,\z)= \Hcal_1(S_2,\z).$ Two Runge subsets $S_1,$ $S_2 \subset \Ncal$ are said to be {\em homeomorphically isotopic} if there exists a homeomorphism $\sigma: S_1 \to S_2$ such that $\sigma_*={\rm Id}_{\Hcal_1(S_1,\z)},$ where $\sigma_*$ is  the induced group morphism on homology. In this case $\sigma$ is said to be an {\em isotopical homeomorphism}. Two Runge domains with finite topology (or two Runge compact regions) in $\Ncal$  are isotopic if and only if they are homeomorphically isotopic.

Let $W$ be a Runge domain of finite topology in $\Ncal$, and let $S$ be a compact Runge subset in $\Ncal.$ $W$ is said to be a {\em tubular neighborhood} of $S$ if $S \subset W$ and $S$ is isotopic to $W.$ In addition, if $\overline{W}$ is a compact region isotopic to $W$ then $\overline{W}$ is said to be a {\em compact tubular neighborhood} of $S.$

\begin{definition}[Admissible set]
A compact subset $S\subset\Ncal$ is said to be admissible if and only if:
\begin{itemize}
\item $M_S:=\overline{S^\circ}$ is a finite collection of pairwise disjoint compact regions in $\Ncal$ with   $\mathcal{ C}^0$ boundary,
\item $C_S:=\overline{S-M_S}$ consists of a finite collection of pairwise disjoint analytical Jordan arcs, 
\item any component $\alpha$ of $C_S$  with an endpoint  $P\in M_S$ admits an analytical extension $\beta$ in $\Ncal$ such that the unique component of $\beta-\alpha$ with endpoint $P$ lies in $M_S,$ and
\item $S$ is Runge.
\end{itemize}
\end{definition}

For any subset $S\subset \Ncal,$ a function $f:S\to\c^n,$ $n\in\n,$ is said to be holomorphic if there exists a open set $U\subset\Ncal$ containing $S$ and a holomorphic function $h:U\to\c$ such that $h|_{S}=f.$

\begin{definition}
Let $S\subset\Ncal$ be an admissible set. A function $f:S\to\c^n,$ $n\in\n,$ is said to be {\em admissible} if $f|_{M_S}$
is holomorphic, and for any component $\alpha$ of $C_S$
and any open analytical Jordan arc $\beta$ in $\Ncal$ containing $\alpha,$  $f$ admits a smooth extension $f_\beta$ to $\beta$
satisfying that $f_\beta|_{U \cap \beta}=h|_{U \cap \beta},$ where $U\subset\Ncal$ is an open domain containing $M_S$ and $h:U\to\c^n$ is a holomorphic extension of $f.$
\end{definition}

Likewise, a complex 1-form $\theta$ of type $(1,0)$ on $S$ is said to be {\em abmissible} if for any closed conformal disc $(W,z)$ in $ \mathcal{N}$ such that $W\cap S$ is admissible then $\theta|_{W \cap S}=g(z) dz$ for an admissible function $g:W \cap S \to \c.$

Given an admissible function $f:S\to \c^n,$ we set $df$ as the vectorial admissible 1-form given by $df|_{M_S}=d (f|_{M_S})$ and $df|_{\alpha \cap W}=(f \circ \alpha)'(x)dz|_{\alpha \cap W},$
where $(W,z=x+i y)$ is a conformal chart on $\Ncal$ such that $\alpha \cap W=z^{-1}(\R \cap z(W)).$

If $f:S \to \c^n$ is admissible, then the $\Ccal^1$-norm of $f$ on $S$ is given by 
$$\|f\|_1=\max_S (\|f\|+ \|df/\omega\|).$$


\section{Main Lemma}\label{sec:lemma}

Set $\pi_1:\c^2\to\c$ the projection  $\pi_1(z,w)=z.$ We will need the following definition:

\begin{definition}[\text{\cite{W2,FW}}]\label{def:exposed}
Let $M\subset \Ncal$ be a Riemann surface possibly with boundary, and let $X:M \to \c^2$ be a proper holomorphic embedding. 
A point $p=(p_1,p_2)$ of the complex curve $\Sigma:=X(M)$ is said to be {\em exposed (with respect to $\pi_1$)} if the complex line $\Lambda_p=\pi_1^{-1}(\pi_1(p))=\{(p_1,w)\;|\; w\in\c\}$ intersects $\Sigma$ only at $p$ and this intersection is transverse, that is to say, $\Lambda_p\cap\Sigma=\{p\}$ and $T_p\Lambda_p\cap T_p\Sigma=\{0\}.$
\end{definition}

The proof of the following technical lemma is inspired by the ideas of Forstneri$\breve{\text{c}}$ and Wold \cite{W2,FW}. Roughly speaking, Lemma \ref{lem:main} below asserts that an embedded bordered Riemman surface in $\c^2$ whose boundary lies outside an Euclidean ball can be perturbed near the boundary in such a way that the boundary of the arising surface lies outside a bigger ball in $\c^2.$ The strength of this lemma is that embeddedness  is preserved in this process.

For any $r>0$ we denote by $\b(r)=\{z\in\c^2\;|\; \|z\|<r\}$ and $\overline{\b}(r)=\{z\in\c^2\;|\; \|z\|\leq r\}.$ 

\begin{lemma}\label{lem:main}
Let $M$ be a Runge compact region in $\Ncal,$ let $X:M\to \c^2$ be a holomorphic embedding and let $r>0$ such that 
\begin{equation}\label{eq:lema}
X(\partial M)\subset \c^2-\overline{\b}(r).
\end{equation}

Then, for any $\xi>0$ and any $\hat{r}>r,$ there exists a Runge compact region $\hat{M}$ on $\Ncal$ and a holomorphic embedding $\hat{X}:\hat{M}\to\c^2$ satisfying that:
\begin{enumerate}[{\rm ({L}.1)}]
\item $\hat{M}$ is a compact tubular neighborhood of $M,$
\item $\|\hat{X}-X\|_1<\xi$ on $M,$
\item $\hat{X}(\partial \hat{M})\subset \c^2-\overline{\b}(\hat{r}),$ and
\item $\hat{X}(\hat{M}-M^\circ)\subset \c^2-\overline{\b}(r).$
\end{enumerate}
\end{lemma}

\begin{proof}
Fix $\xi_0\in ]0,\xi[$ so that
\begin{equation}\label{eq:lema'}
X(\partial M)\subset \c^2-\overline{\b}(r+\xi_0),
\end{equation}
see \eqref{eq:lema}. Take $\epsilon_0>0$ to be specified later.

We begin exposing boundary points as in \cite{FW}. 

Since we are assuming that $X$ holomorphically extends beyond $M,$ there exists a Runge compact region $N_1$ on $\Ncal$ and a holomorphic embedding $Y_1:N_1\to\c$ such that
\begin{enumerate}[{\rm ({a}.1)}]
\item $N_1$ is a compact tubular neighborhood of $M,$
\item $Y_1|_M=X,$ and
\item $Y_1(N_1-M^\circ)\subset \c^2-\overline{\b}(r+\xi_0).$
\end{enumerate}

Write $\partial M=\cup_{j=1}^m C_j,$ where $\{C_j\}_{j=1}^m$ are the connected components of $\partial M.$ Choose a point $a_j\in C_j$ and an analytic Jordan arc $\gamma_j\subset N_1^\circ-M^\circ$ with initial point $a_j,$ otherwise disjoint from $\partial M$ and such that the intersection of $\gamma_j$ and $C_j$ is transverse, $\forall j=1,\ldots,m.$ Take the arcs $\{\gamma_j\}_{j=1,\ldots,m}$ so that $M\cup(\cup_{j=1}^m \gamma_j)$ is admissible. Let $b_j$ denote the other endpoint of $\gamma_j,$ and let $U\subset N_1^\circ$ be a compact tubular neighborhood of $M$ such that $b_j\notin U,$ $U\cup(\cup_{j=1}^m \gamma_j)$ is admissible and $\tilde{\gamma}_j:=\gamma_j\cap U$ is a Jordan arc with an endpoint at $a_j,$ $j=1,\ldots,m.$

On the other hand, consider  pairwise disjoint smooth regular Jordan arcs $\{\lambda_j\;| \;j=1,\ldots,m\}$ in $\c^2$  such that 
\begin{enumerate}[{\rm ({b}.1)}]
\item $Y_1(a_j)$ is an endpoint of $\lambda_j,$ $Y_1(\tilde{\gamma}_j)\subset \lambda_j$ and $(\lambda_j-Y_1(\tilde{\gamma}_j))\cap Y_1(U)=\emptyset,$
\item $\lambda_j\subset \c^2-\overline{\b}(r+\xi_0),$ 
\item the other endpoint $p_j$ of $\lambda_j$ satisfies $\Lambda_{p_j}\cap (Y_1(M)\cup (\cup_{i=1}^m \lambda_i))=\{p_j\}$ and $T_{p_j}\Lambda_{p_j}\cap T^\c_{p_j}\lambda_j=\{0\},$ where $T^\c_{p_j}\lambda_j$ is the complexification of the real tangent line to $\lambda_j$ at $p_j,$ and
\item $|\pi_1(p_j)|>r+\xi_0,$ 
\end{enumerate} 
for all $j\in \{1,\ldots,m\},$ see Figure \ref{fig:exposing}. Notice that (b.3) is a generalization of Definition \ref{def:exposed}. Item (b.2) is possible since (a.3) holds.
\begin{figure}[ht]
    \begin{center}
    \scalebox{0.45}{\includegraphics{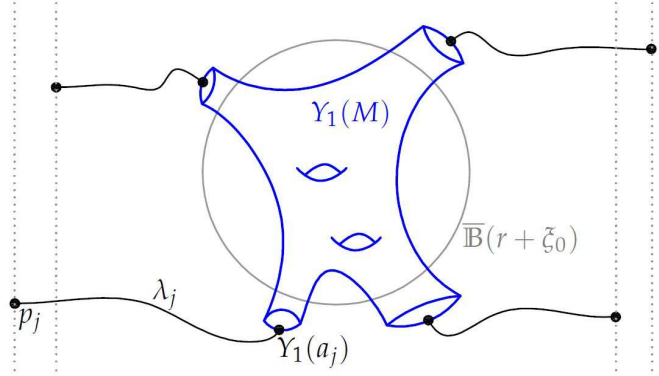}}
        \end{center}
        \vspace{-0.5cm}
\caption{The arcs $\lambda_j.$}\label{fig:exposing}
\end{figure}

Consider an admissible embedding $\hat{Y}_1:U\cup (\cup_{j=1}^m \gamma_j)\to \c^2$ such that
\begin{enumerate}[{\rm ({c}.1)}]
\item $\hat{Y}_1|_U=Y_1,$ and 
\item $\hat{Y}_1(\gamma_j)=\lambda_j,$ $j=1,\ldots,m.$ In particular, $\hat{Y}_1(b_j)=p_j.$
\end{enumerate}

By Mergelyan's Theorem (see for instance \cite[Theorem 3.2]{Fo}), we can find a Runge compact region $N_2$ and a holomorphic embedding $Y_2:N_2\to\c^2$ such that
\begin{enumerate}[{\rm ({d}.1)}]
\item $N_2$ is a compact tubular neighborhood of $U$ (hence, of $M$), with $\gamma_j\subset N_2\subset N_1^\circ$ and $b_j\in \partial N_2,$ $j=1,\ldots,m,$
\item $\|Y_2-\hat{Y}_1\|_1<\epsilon_0$ on $U\cup (\cup_{j=1}^m \gamma_j),$
\item $Y_2(N_2-M^\circ)\subset\c^2-\overline{\b}(r+\xi_0),$ and
\item $Y_2(b_j)=\hat{Y}_1(b_j)=p_j$ is an exposed point for $Y_2(N_2).$
\end{enumerate}
Notice that (d.3) can be guaranteed from (a.3), (b.2), (c.1) and (c.2). Property (d.4) is possible thanks to (b.3) (see  \cite[Theorem 4.2]{FW} for more details).

The second step in the proof of Lemma \ref{lem:main} consists of pushing $Y_2(\partial N_2)$ out of $\overline{\b}(\hat{r}).$ Now we are inspired by \cite{W2} and \cite[Theorem 5.1]{FW}.

Write $\partial N_2=\cup_{j=1}^m \Gamma_j,$ where $\{\Gamma_j\}_{j=1}^m$ are the connected components of $\partial N_2.$ Set 
\begin{equation}\label{eq:g}
g:\c^2\to\c\times\overline{\c},\quad g(z,w)=\left(z,w+\sum_{j=1}^m \frac{\alpha_j}{z-\pi_1(Y_2(b_j))}\right),
\end{equation}
where the coefficients $\alpha_j\in\c-\{0\}$ are chosen so that the following assertions hold.
\begin{enumerate}[{\rm ({e}.1)}]
\item $\pi_2$ maps the curve $\mu_j:=(g\circ Y_2)(\Gamma_j-\{b_j\})\subset\c^2$ into an unbounded curve $\delta_j\subset\c,$ and $\pi_2:\mu_j\to\delta_j$ is a diffeomorphism near infinity, where $\pi_2:\c^2\to\c$ is given by $\pi_2(z,w)=w.$
\item $\overline{\d}(\rho)\cup (\cup_{j=1}^m \delta_j)$ is Runge in $\c$ for any large enough $\rho\in\r,$ where $\overline{\d}(\rho)=\{z\in\c\;|\;|z|\leq\rho\}.$
\item $\|g\circ Y_2-Y_2\|_1<\epsilon_0$ on $M.$ 
\item $(g\circ Y_2)((N_2-\{b_j\}_{j=1,\ldots,m})-M^\circ)\subset \c^2-\overline{\b}(r+\xi_0).$
\end{enumerate}
This can be guaranteed by a careful choice of the argument of the complex number $\alpha_j,$ while $|\alpha_j|$ must be chosen as small as needed, $j=1,\ldots,m.$ To achieve properties (e.3) and (e.4), we argue as follows. First, fix pairwise disjoint small open discs $W_j\subset \Ncal,$ $j=1,\ldots,m,$ such that $b_j \in W_j$ and 
\begin{equation}\label{eq:nofacil}
|\pi_1(Y_2(W_j \cap N_2))|>r+\xi_0\;\text{ for all $j,$}
\end{equation}
see (b.4). Then choose $|\alpha_j|,$ $j=1,\ldots,m,$ small enough so that $(g \circ Y_2)(N_2-(M^\circ \cup_j W_j))\subset \c^2-\overline{\b}(r+\xi_0)$ (see (d.3)) and $\|g\circ Y_2-Y_2\|_1<\epsilon_0$ on $M.$ As $\pi_1 \circ g=\pi_1,$ then \eqref{eq:nofacil} gives that $(g \circ Y_2)(N_2\cap W_j)\subset \c^2-\overline{\b}(r+\xi_0)$ as well.

Label $W:=N_2-\{b_j\}_{j=1,\ldots,m},$ set $Z:W\to\c^2,$ $Z:=g\circ Y_2|_W,$ and note that $Z$ is a well defined holomorphic embedding thanks to (d.4). Furthermore, $Z$ has the following property: there exists a compact polynomially convex $K_0\subset Z(W)$ in $\c^2$ such that $K:=K_0\cup \overline{\b}(r+\xi_0)$ is polynomially convex and $Z(M)\subset K\subset \c^2-Z(\partial W)$ (see (e.4) and the proof of Theorem 5.1 in \cite{FW}). Moreover there exists a holomorphic automorphism $\phi$ of $\c^2$ such that
\begin{enumerate}[{\rm ({f}.1)}]
\item $Z(M)\cup\overline{\b}(r+\xi_0)\subset K\subset \c^2-Z(\partial W),$ notice that $\partial W=\partial N_2-\{b_j\}_{j=1,\ldots,m},$
\item $\|\phi-{\rm Id}_{\c^2}\|<\epsilon_0$ on $K,$ and $\|\phi\circ Z -Z\|_1<\epsilon_0$ on $M,$ and
\item $(\phi\circ Z)(\partial W)\subset \c^2-\overline{\b}(\hat{r}).$
\end{enumerate}
Such $\phi$ is constructed in \cite{W1} from (e.1) and (e.2), see also the proof of Theorem 5.1 in \cite{FW}.

Define $\hat{X}:W\to\c^2,$ $\hat{X}:=\phi\circ Z,$ and let us check that $\hat{X}$ {\em almost} satisfies the conclusion of Lemma \ref{lem:main}.
\begin{itemize}
\item $W^\circ$ is an open tubular neighborhood of $M.$ See (d.1) and the definitions of $U$ and $W.$

\item $\|\hat{X}-X\|_1<\xi$ on $M.$ Indeed, use (a.2), (d.2), (e.3), (f.1) and (f.2) and assume that $\epsilon_0$ was chosen small enough from the beginning.

\item $\hat{X}(\partial W)\subset \c^2-\overline{\b}(\hat{r}).$ See (f.3).

\item $\hat{X}(W-M^\circ)\subset \c^2-\overline{\b}(r).$ Indeed, if $\epsilon_0$ is chosen small enough from the begining, then taking into account (f.1), (f.2) and that $\phi$ is bijective, we conclude that $\phi(\c^2-\overline{\b}(r+\xi_0))\subset  \c^2-\phi(\overline{\b}(r)).$ Then use (e.4).
\end{itemize}

Taking into account these properties of $\hat{X},$ we finish by setting $\hat{M}$ as a suitable shrinking of $W$ satisfying (L.1).
The proof is done.
\end{proof}


\section{Main Theorem}\label{sec:theorem}

We will need the following
\begin{definition}\label{def:emb}
Let $K$ be a compact subset of $\Ncal,$  let $f:K\to\c^2$ be a topological embedding and let $n\in\n.$ We define
\[
\Psi(K,f,n):=\left.\frac{1}{2n^2} \inf \right\{ \|f(p)-f(q)\|\;\left|\; p, q\in K,\; \dsf(p,q)>\frac1{n}  \right\},
\]
where $\dsf(\cdot,\cdot)$ means distance in the Riemannian surface $(\Ncal,\omega),$ see Remark \ref{re:fun}.
Notice that $\Psi(K,f,n)>0.$
\end{definition}

Now we can state and prove the main result of this paper. 


\begin{theorem}\label{th:main}
Let $N$ be a Runge compact region on $\Ncal$ and  let $Y:N\to\c^2$ be a holomorphic embbeding. Assume that
\begin{equation}\label{eq:theorem}
Y(\partial N)\subset \c^2-\overline{\b}(s)
\end{equation}
for a positive $s.$

Then, for any $\epsilon>0,$ there exist an open domain $\Mcal\subset \Ncal$ and a proper holomorphic embedding $\Xcal:\Mcal\to \c^2$ satisfying
\begin{enumerate}[{\rm ({T}.1)}]
\item $N\subset\Mcal,$ $\Mcal$ is Runge and isotopic to $\Ncal,$
\item $\|\Xcal-Y\|_1<\epsilon$ on $N,$ and
\item $\Xcal(\Mcal-N^\circ)\subset \c^2-\overline{\b}(s).$
\end{enumerate} 
\end{theorem}

\begin{proof}
Consider an exhaustion $\{M_j\}_{j\in\n}$ of $\Ncal$ by Runge compact regions so that $M_1=N,$ and $M_{j-1}\subset M_j^\circ$ and the Euler characteristic $\chi(M_j-M_{j-1}^\circ)\in\{-1,0\}$ for all $j\geq 2$ (if $\Ncal$ has finite topology then $\chi(M_j-M_{j-1}^\circ)=0$ for any large enough $j$).

Since $Y(\partial N)$ is compact, equation \eqref{eq:theorem} guarantees the existence of $s_0>s$ such that
\begin{equation}\label{eq:theorem'}
Y(\partial N)\subset \c^2-\overline{\b}(s_0).
\end{equation}

Take $\epsilon_0>0$ with
\begin{equation}\label{eq:ep0}
\epsilon_0<\min\{\epsilon,s_0-s\},
\end{equation}
to be specified later. 

\begin{claim} There exists a sequence $\{\Xi_j\}_{j \in \n}:=\{(N_j,\sigma_j,Y_j,\epsilon_j)\}_{j\in\n},$ where
\begin{itemize}
\item $N_j$ is a Runge compact region on $\Ncal$ isotopic to $M_j,$
\item $\sigma_j:N_j\to M_j$ is an isotopic homeomorphism,
\item $Y_j:N_j\to\c^2$ is a holomorphic embedding, and
\item $\epsilon_j>0,$ $j\in\n,$
\end{itemize}
such that 
\begin{enumerate}[{\rm (A{$_j$})}]
\item $N_{j-1}\subset N_j^\circ$ and  $\sigma_j|_{N_{j-1}}=\sigma_{j-1},$  
\item $\epsilon_j<\min\{\epsilon_0/2^j\,,\,\Psi(N_{j-1},Y_{j-1},j)\,,\,\epsilon_{j-1}\,,\,\varrho_{j-1}\},$ where
\[
\varrho_{j-1}:=\frac1{2^j}\min\left\{ \min_{N_k} \big\| \frac{dY_{k}}{\omega} \big\|\;\big|\; k=1,\ldots,j-1\right\} >0,
\] 
\item $\|Y_j-Y_{j-1}\|_1<\epsilon_j$ on $N_{j-1},$ 
\item $Y_j(\partial N_j)\subset \c^2-\overline{\b}(s_0+j-1),$ and 
\item $Y_j(N_j-N_{j-1}^\circ)\subset \c^2-\overline{\b}(s_0+j-2).$
\end{enumerate}
\end{claim}
\begin{proof}
The sequence is constructed inductively. Set $\Xi_1:=(N,{\rm Id}|_{N},Y,\epsilon_1),$ where $\epsilon_1<\epsilon_0/4.$ Equation \eqref{eq:theorem'} gives property (D$_1$) whereas properties (A$_1$), (B$_1$), (C$_1$) and (E$_1$) do not make sense. 


To prove the inductive step, assume that $\Xi_1,\ldots,\Xi_{j-1}$ are already constructed satisfying the required properties and let us construct $\Xi_j,$ $j\geq 2.$

We need to distinguish two cases depending on $\chi(M_j-M_{j-1}^\circ).$

\noindent $\bullet$ {\bf Case 1.} Assume $\chi(M_j-M_{j-1}^\circ)=0.$ Apply Lemma \ref{lem:main} to the data
\[
M=N_{j-1},\quad X=Y_{j-1},\quad r=s_0+j-2,\quad \xi=\epsilon_j\quad\text{and}\quad \hat{r}=s_0+j-1,
\]
where $\epsilon_j$ is any positive satisfying (B$_j$). Observe that the lemma can be applied thanks to property (D$_{j-1}$). Then we set $\Xi_j:=(N_j=\hat{M},\sigma_j,Y_j=\hat{X},\epsilon_j),$ where $\hat{M}$ and $\hat{X}$ are the data arising from the lemma and $\sigma_j:N_j\to M_j$ is any homeomorphism with $\sigma_j|_{N_{j-1}}=\sigma_{j-1}.$ Properties (A$_j$), (C$_j$), (D$_j$) and (E$_j$) directly follow from (L.1), (L.2), (L.3) and (L.4) of Lemma \ref{lem:main}, respectively.

\noindent $\bullet$ {\bf Case 2.} Assume $\chi(M_j-M_{j-1}^\circ)=-1.$ First of all, fix $\epsilon_j>0$ satisfying (B$_j$). 
Take a Runge compact region $R_1$ and a holomorphic embedding $Z_1:R_1\to\c^2$ such that
\begin{enumerate}[{\rm ({a}.1)}]
\item $R_1$ is a compact tubular neighborhood of $N_{j-1},$
\item $Z_1|_{N_{j-1}}=Y_{j-1},$ and
\item $Z_1(R_1-N_{j-1}^\circ)\subset \c^2-\overline{\b}(s_0+j-2).$
\end{enumerate}
(Recall that $Y_{j-1}$ extends holomorphically beyond $N_{j-1}$ in $\Ncal.$)
Consider a smooth Jordan curve $\hat{\alpha}\in \Hcal_1(M_j,\z)-\Hcal_1(M_{j-1},\z)$ contained in $M_j^\circ$ and intersecting $M_j-M_{j-1}^\circ$ in a Jordan arc $\alpha$ with  endpoints $a,$ $b$ in $\partial M_{j-1}$ and otherwise disjoint from $M_{j-1}.$ Since $M_{j-1}$ and $M_j$ are Runge and $\chi(M_j-M_{j-1}^\circ)=-1,$ then  $\Hcal_1(M_j,\z)=\Hcal_1(M_{j-1}\cup \alpha,\z)$ and $M_{j-1} \cup \alpha$ is Runge as well. Take an analytic Jordan arc $\gamma\subset  \Ncal-N_{j-1}^\circ$ with endpoints $\sigma_{j-1}^{-1}(a),$ $\sigma_{j-1}^{-1}(b)$ in  $\partial N_{j-1},$ otherwise disjoint from $N_{j-1},$ transversally intersecting $\partial N_{j-1}$ and such that $N_{j-1}\cup\gamma$ is admissible. Take also an isotopic homeomorphism $\varsigma:N_{j-1} \cup \gamma \to M_{j-1} \cup \alpha$ so that  $\varsigma|_{N_{j-1}}=\sigma_{j-1}$ and $\varsigma(\gamma)=\alpha.$

On the other hand, consider in $\c^2$ a smooth regular Jordan arc $\lambda$ agreeing with $Z_1(\gamma\cap R_1)$ near the endpoints $Z_1(\sigma_{j-1}^{-1}(a))$ and $Z_1(\sigma_{j-1}^{-1}(b)),$ and such that
\begin{enumerate}[{\rm ({b}.1)}]
\item $(\lambda-Z_1(\gamma\cap R_1))\cap Z_1(N_{j-1})=\emptyset,$ and
\item $\lambda\subset \c^2-\overline{\b}(s_0+j-2).$
\end{enumerate}
This choice of $\lambda$ is possible thanks to property (a.3). Consider the admissible embedding $\hat{Z}_1:N_{j-1}\cup\gamma\to \c^2$ given by $\hat{Z}_1|_{N_{j-1}}=Z_1$ and $\hat{Z}_1(\gamma)=\lambda.$ Mergelyan's Theorem provides Runge a compact region $R_2$ and a holomorphic embedding $Z_2:R_2\to\c^2$ satisfying that
\begin{enumerate}[{\rm ({c}.1)}]
\item $R_2$ is a compact tubular neighborhoods of $N_{j-1}\cup\gamma,$ 
\item $\|Z_2-\hat{Z}_1\|_1<\epsilon_j/2$ on $N_{j-1}\cup\gamma,$ and
\item $Z_2(R_2-N_{j-1}^\circ)\subset \c^2-\overline{\b}(s_0+j-2).$
\end{enumerate}

Since $Z_2(R_2-N_{j-1}^\circ)$ is compact, (c.3) implies the existence of $\varepsilon\in]0,\epsilon_j/2[$ small enough so that
\begin{equation}\label{eq:ep*}
Z_2(R_2-N_{j-1}^\circ)\subset \c^2-\overline{\b}(s_0+j-2+\varepsilon).
\end{equation}

Set $\Xi_j:=(N_j=\hat{M},\sigma_j,Y_j=\hat{X},\epsilon_j),$ where $\hat{M}$ and $\hat{X}$ are the data arising from Lemma \ref{lem:main} applied to the data
\[
M=R_2,\quad X=Z_2,\quad r=s_0+j-2+\varepsilon,\quad \xi=\varepsilon\quad\text{and}\quad \hat{r}=s_0+j-1,
\]
where $\sigma_j:N_j\to M_j$ is any homeomorphism with $\sigma_j|_{N_{j-1}\cup\gamma}=\varsigma.$ Observe that the lemma can be applied thanks to (c.3). Property (A$_j$) follows from (a.1), (c.1) and Lemma \ref{lem:main}-(L.1). Property (C$_j$) is implied by (a.2), (c.2) and Lemma \ref{lem:main}-(L.2). Item (L.3) in Lemma \ref{lem:main} gives (D$_j$). Finally, to check (E$_j$) consider a point $p\in N_j-N_{j-1}^\circ$ and let us distinguish cases. If $p\in N_j-R_2^\circ$ then Lemma \ref{lem:main}-(L.4) gives $Y_j(p)\in \c^2-\overline{\b}(s_0+j-2)$ and we are done. Otherwise $p\in R_2-N_{j-1}^\circ,$ and in this case Lemma \ref{lem:main}-(L.2) and equation \eqref{eq:ep*} guarantee that $Y_j(p)\in \c^2-\overline{\b}(s_0+j-2)$ as well.

This concludes the construction of the sequence $\{\Xi_j\}_{j\in\n}$ satisfying the desired properties.
\end{proof}

Set $\Mcal:=\cup_{j\in\n} N_j$ and $\sigma:\Mcal \to \Ncal,$ $\sigma|_{N_j}=\sigma_j.$ Since $\{M_j\}_{j\in\n}$ is an exhaustion of $\Ncal$ by Runge compact regions and $\sigma_j$ is an isotopic homeomorphism for all $j,$ then  $\sigma$ is an isotopic homeomorphism as well and statement (T.1) holds.

Properties (B$_j$) and (C$_j$), $j\in\n,$ imply that the sequence of holomorphic maps $\{Y_j\}_{j\in\n}$ uniformly converges on compact subsets of $\Mcal$ to a holomorphic map $\Xcal:\Mcal\to\c^2$ satisfying 
\begin{equation}\label{eq:<ep0}
\|\Xcal-Y\|_1<\epsilon_0\quad\text{on }N.
\end{equation}
(Recall that $Y_1=Y$ and $N_1=N$). This implies (T.2) (see equation \eqref{eq:ep0}).

Let us check (T.3). Take $p\in \Mcal-N^\circ.$ Then, there exists $j\geq 2$ such that $p\in N_j-N_{j-1}^\circ$ and, by properties (C$_j$) and (E$_j$), $\Xcal (p)\in\c^2-\overline{\b}(s_0+j-2-\epsilon_0)\subset\c^2-\overline{\b}(s),$ see \eqref{eq:ep0}.

To check that $\Xcal$ is injective we have to work a little further. Take $p,q\in \Mcal,$ $p\neq q,$ and let us prove that $\Xcal(p)\neq \Xcal(q).$ Indeed, consider a large enough $j_0\in\n$ so that  
$\{p,q\}\subset N_{j}$ and $\dsf(p,q)>1/j,$ $\forall j\geq j_0.$
Then, for any $j>j_0,$ from properties (B$_j$) and (C$_j$), one has
\begin{eqnarray*}
\|Y_{j-1}(p)-Y_{j-1}(q)\| & \leq & \|Y_{j-1}(p)-Y_{j}(p)\| + \|Y_{j}(p)-Y_{j}(q)\| + \|Y_{j}(q)-Y_{j-1}(q)\| \\
& < & 2\epsilon_j + \|Y_{j}(p)-Y_{j}(q)\| \\
& \leq & \frac{1}{j^2} \|Y_{j-1}(p)-Y_{j-1}(q)\|+ \|Y_{j}(p)-Y_{j}(q)\|,
\end{eqnarray*}
see Definition \ref{def:emb}. Therefore, $\|Y_j(p)-Y_j(q)\|\geq ( 1-{1}/{j^2}) \|Y_{j-1}(p)-Y_{j-1}(q)\|,$ $\forall j>j_0,$
and so
\[
\|Y_{j_0+i}(p)-Y_{j_0+i}(q)\|\geq \|Y_{j_0}(p)-Y_{j_0}(q)\|\cdot\prod_{j=j_0+1}^{j_0+i}\left( 1-\frac{1}{j^2}\right),\quad \forall i\in\n.
\]
Taking limits in the above inequality as $i\to\infty$ we obtain that $\|\Xcal(p)-\Xcal(q)\|\geq \frac12\|Y_{j_0}(p)-Y_{j_0}(q)\|>0$ (recall that $Y_{j_0}$ is an embedding) and we are done.

Let us check that $\Xcal:\Mcal\to\c^2$ is proper. Consider a compact subset $K\subset\c^2.$ It suffices to prove that $\Xcal^{-1}(K)$ is compact in $\Mcal.$ Take $j_0\in\n$ large enough so that $K\subset \overline{\b}(s_0+j_0-2-\epsilon_0).$ On the other hand, properties (B$_j$) and (E$_j$) give that $\Xcal(N_j-N_{j-1}^\circ)\subset \c^2-\overline{\b}(s_0+j_0-2-\epsilon_0)$ for any $j\geq j_0.$ Hence $\Xcal^{-1}(K)\subset N_{j_0-1}$ which is compact in $\Mcal,$ and we are done.

Finally, let us check that $\Xcal$ is an immersion, hence an embedding. Take $p\in \Mcal$ and $j_0\in\n$ such that $p\in N_j$ $\forall j\geq j_0.$ Then 
\begin{eqnarray*}
\|d\Xcal/\omega\|(p) & \geq & \|dY_{j_0}/\omega\|(p)-\sum_{j>j_0}\|Y_j-Y_{j-1}\|_1\geq \|dY_{j_0}/\omega\|(p)-\sum_{j>j_0}\epsilon_j
\\
&\geq & \|dY_{j_0}/\omega\|(p)-\sum_{j>j_0}\varrho_{j-1}\geq \|dY_{j_0}/\omega\|(p)\big(1-\sum_{j>j_0}\frac1{2^j}\big)\geq \frac12\|dY_{j_0}/\omega\|(p)>0,
\end{eqnarray*}
where we have used (B$_j$), $j> j_0.$ The proof of Theorem \ref{th:main} is done.
\end{proof}

Main Theorem in the introduction easily follows from Theorem \ref{th:main}. Indeed, let $\Ncal$ be an open Riemann surface, let $N$ be a conformal compact disc on $\Ncal$ and let $Y:N\to\c^2$ be a holomorphic embedding with $Y(\partial N)\subset \c^2-\overline{\b}(1).$ Then Theorem \ref{th:main} provides an open domain $\Mcal\subset\Ncal$ homeomorphic to $\Ncal$ and a proper holomorphic embedding $\Xcal:\Mcal\to\c^2.$ Furthermore, if we substitute $\Ncal$ for any hiperbolic isotopic subdomain of $\Ncal,$ the arising domain $\Mcal$ is hyperbolic as well.


\end{document}